\documentclass[a4paper,12pt]{article}

\usepackage{amsmath,amssymb,amstext,amsfonts,latexsym,amsthm}
\usepackage{dsfont,mathptmx}
\usepackage{graphicx,hyperref,color}
\usepackage{enumerate}
\newtheorem{theorem}{Theorem}
\newtheorem{remark}{Remark}
\newtheorem{lemma}{Lemma}
\newtheorem{corollary}{Corollary}

\newcommand{\CC}{\mathds{C}}
\newcommand{\ZZp}{\mathds{Z}_{+}}
\newcommand{\RRp}{\mathds{R}_{+}}
\newcommand{\NN}{\mathds{N}}

\newcommand{\DD}{\mathbf{D}}
\newcommand{\arc}{\mathbf{A}}
\newcommand{\dsty}{\displaystyle}
\newcommand{\capa}[1]{\mathop{\mathrm{cap}\left(#1 \right)}}
\newcommand{\wlim}{\mathop{\mathrm{w\mbox{-}lim}}}
\newcommand{\supp}[1]{\mathop{\mathrm{supp}\left(#1 \right)}}

\title{Iterated integrals and Borwein-Chen-Dilcher polynomials}

\author{Manuel Bello-Hern\'{a}ndez\thanks{mbello@unirioja.es}\\Universidad de La Rioja, Spain  \and
H\'{e}ctor Pijeira-Cabrera\thanks{hpijeira@math.uc3m.es}\\ Universidad Carlos III de Madrid, Spain
\and Daniel  Rivero-Castillo\thanks{darc12@gmail.com} \\ Universidad Polit\'{e}cnica de Madrid, Spain }

\date{}
\begin{document}

\maketitle

\begin{abstract}
We study the zero location and the asymptotic behavior of iterated integrals of polynomials. Borwein-Chen-Dilcher's polynomials play an important role in this issue. For these polynomials we find their strong asymptotics and give the limit measure of their zero distribution. We apply these results to describe the zero asymptotic distribution  of iterated integrals of ultraspherical polynomials with parameters $(2\alpha+1)/2$, $\alpha\in \ZZp$.
\end{abstract}

\section{Introduction and main results}

Several problems require the  location of the zeros of a polynomial  in areas such as   numerical analysis, approximation theory, differential equations, and complex dynamics. The zeros  of a polynomial can represent equilibrium  points in a certain force field, geometric points of certain curves, critical  points and so on (see \cite{KhPePuSa11,RahSch02,Ran95,She02}, and the references therein).  The objective of this paper  is the study of some algebraic and asymptotic properties of the zeros of iterated integrals of  polynomials.

Given a monic polynomial $p_n$ of degree $n$, $\lambda\in\CC$, and $m\in\ZZp$, its $m-$fold integral
\begin{equation}
\label{IntPol}
I_{m,\lambda}(p_n)(z):=\frac{(n+m)!}{n!}\int_\lambda^z\int_\lambda^{t_{m-1}}\ldots\int_\lambda^{t_1}p_n(t_0)\, dt_0\ldots dt_{m-2}dt_{m-1}
\end{equation}
defines  a monic polynomial of degree $n+m$ for which the derivatives of order $j$, $0\le j\le m-1$, at $\lambda$  are zero and
\begin{equation}
\label{IntPol2}
(I_{m,\lambda}(p_n))^{(m)}(z)=\frac{(n+m)!}{n!}p_n(z),
\end{equation}
where $m \in \ZZp$ and $\lambda \in \CC$. Of course, $I_{0,\lambda}(p_n):=p_n$. When $\lambda=0$, for simplicity of notation, let $I_m(p_n):=I_{m,0}(p_n)$ . The interchange of the order of integration or  integration by parts yields
\begin{align*}
 I_{m,\lambda}(p_n)(z)= & \frac{(n+m)!}{(m-1)!n!}\int_{\lambda}^{z} (z-s)^{m-1} p_n(s) \;ds \\
  = & \frac{(n+m)!z^{m}}{(m-1)!n!}\int_{\lambda/z}^{1} (1-t)^{m-1} p_n(zt) \;dt.
\end{align*}

Let $Q_{n,m}$ be the polynomials of degree $n$ given by
\begin{align}
\label{polQnm}
Q_{n,m}(z):= &\frac{(n+m)!}{n!}\sum_{k=0}^n\binom{n}{k}\frac{k!}{(m+k)!}z^k=\sum_{k=0}^n\binom{n+m}{k+m}z^k\\ = &\frac{(n+m)!}{(m-1)!n!}\int_{0}^{1} (1-t)^{m-1} (1+zt)^n\;dt=\frac{(n+m)!}{n!z^m}I_m((1+z)^n).
\end{align}
Borwein, Chen and Dilcher (\cite[Th. 1]{BoChDi95}) prove that the zeros of $Q_{n,n+1}$  are dense in the curve
$$
\Gamma:=\left\{z\in\CC:\left|\frac{(z+1)^2}{4z}\right|=1 \textit{ and }|z|\ge1\right\}.
$$
and these are the only limit points of the zeros. Using this result, they give estimations for the radius of a disc containing the zeros of $I_n(p_n)$ for every $n$. Moreover, they also obtain the curve to which the zeros of the $n-$fold integral of the $n-$th Legendre polynomial converge, as $n$ goes to infinity.
We call $Q_{n,m}$ \textit{Borwein-Chen-Dilcher polynomials}.

We give strong asymptotics for $Q_{n,n}$. This allows us to characterize the measure which describes their zero distribution. The steepest descent method is used to obtain this estimation. For this description, let us consider two regions induced by  $\Gamma$:
\begin{align*}
\mathcal{E}_1:= & \left\{z\in\CC:\left|\frac{(z+1)^2}{4z}\right|>1,|z|>1\right\}, \\
 \mathcal{E}_2:= & \left\{z\in\CC:\left|\frac{(z+1)^2}{4z}\right|<1\right\}\bigcup \left\{z\in\CC:\left|\frac{(z+1)^2}{4z}\right|>1,|z|<1\right\}.
\end{align*}

\begin{theorem}\label{TeoAsintPnn} We have
\begin{equation}
\label{AsintQnn}
Q_{n,n}(z)=\left\{
\begin{array}{ll}
\frac{(z+1)^{2n}}{z^n} (1+o(1))&\textit{if }z\in\mathcal{E}_1,\\ \\
\frac{2^{2n}}{\sqrt{\pi n}(1-z)}(1+o(1))& \textit{if }z\in \mathcal{E}_2,
\end{array}
\right.
\end{equation}
uniformly,  as $n\to\infty$, on compact subsets of each  stated domain\footnote{Here we state results for $m=n$ but analogous statements hold for $m=n+j$ with $j$ a fixed integer.}.
\end{theorem}

A consequence of the above result is the zero distribution of $Q_{n,n}$. Define
$$
\nu[Q_{n,n}]:=\frac{1}{n}\sum_{Q_{n,n}(z)=0}\delta_\zeta,
$$
its weak-$*$ limit is given in terms of  the equilibrium measure of $\Gamma$.  For each Borel set $B\subset \CC$, $\dsty \mu_\Gamma(B)=m(\varphi(B\cap \Gamma)),$
where $dm=d\theta/(2\pi)$, the normalized arc-length on $\partial\DD_1$, and $\varphi(z)=\frac{(z+1)^2}{4z}$.

\begin{corollary}\label{CorZeroDistQ} It holds $\dsty \wlim_{n\to\infty}\nu[Q_{n,n}]=\mu_\Gamma,$ where $\mu_\Gamma$ is the equilibrium measure on $\Gamma$. The support of  $\mu_\Gamma$ is  $\Gamma$. The measure $\mu_\Gamma$ is the pre-image  of the normalized arc-length on $\partial\DD_1$ under the mapping $\varphi(z)=\frac{(z+1)^2}{4z}$ from $\Gamma$ to $\partial\DD_1$. We have $\capa{\Gamma}=4$. Moreover, for $n$ large enough and all $m\in\ZZp$ the zeros of $Q_{n,m}$ are inside  $\Gamma$.

\end{corollary}

Another consequence of Theorem \ref{TeoAsintPnn} is the strong asymptotics and zeros distribution of iterated integrals of ultraspherical polynomials. Let $\alpha\in\ZZp$ be a  positive integer and let
$\widehat{P}_{n}^{(\alpha+1/2)}$ be monic ultraspherical polynomials with parameter $\alpha+1/2$. The polynomials $\widehat{P}_{n}^{(1/2)}$ are monic Legendre polynomials.

\begin{theorem}\label{TeoAsintIntUlt} Let $\alpha\in\ZZp$ be a natural number. It holds
\begin{equation*}
\label{AsintI2nUlt}
I_{2n}\left(\widehat{P}_{2n}^{(\alpha+1/2)}\right)(z)=\left\{
\begin{array}{ll}
(z^2-1)^{2n-\alpha} z^{2\alpha}(1+o(1))&\textit{if }z\in\mathcal{F}_1,\\ \\
\frac{(-1)^{n}2^{2n-\alpha}z^{2n}}{\sqrt{\pi n}(1+z^2)}(1+o(1))& \textit{if }z\in \mathcal{F}_2,
\end{array}
\right.
\end{equation*}
\begin{equation*}
\label{AsintI2n+1Ult}
I_{2n+1}\left(\widehat{P}_{2n+1}^{(\alpha+1/2)}\right)(z)=\left\{
\begin{array}{ll}
(z^2-1)^{2n+1-\alpha} z^{2\alpha}(1+o(1))&\textit{if }z\in\mathcal{F}_1,\\ \\
\frac{(-1)^{n}2^{2n+1-\alpha}z^{2n+2}}{\sqrt{\pi n}(1+z^2)}(1+o(1))& \textit{if }z\in \mathcal{F}_2,
\end{array}
\right.
\end{equation*}
as $n\to\infty$, where $\mathcal{F}_j$ are the pre-image of $\mathcal{D}_j$ under the transformation $\varphi_1(z)=-z^2$. Moreover, its zero distribution\footnote{The zeros of $I_n(\widehat{P}_n)$ different from $0$.} $\beta$ is supported on $\varphi_1(\Gamma)$ and for each Borel set $E$ this measure satisfies
$$
\beta(E)=\mu_{\Gamma}(\varphi_1(E)).
$$
\end{theorem}

Ultimately, the behavior of $Q_{n,n}$ is the same as that obtained by integrating polynomials a number of times which does not change with the integrand (see Section \ref{SectProofTh1}). In order to set our result in this context, let us fix some notations. Let $\arc$ be a Jordan rectifiable arc in $\CC$ and $\Omega:=\overline{\CC}\setminus\arc$. Given $r\in(0,\infty)$, $\DD_r:=\{z\in\CC:|z|<r\}$, $\partial\DD_r:=\{z\in\CC:|z|=r\}$,  $\overline{\DD}_r^c:=\{z\in\CC:|z|>r\}$, and $\DD_r(z_0):=\{z\in\CC:|z-z_0|<r\}$ for $z_0\in\CC$.
Let $\tau$ denote the conformal mapping of $\Omega$ onto $\overline{\DD}_1^c$ such that $\tau'(\infty):=\capa{\arc}:=\lim_{z\to\infty}\frac{\tau(z)}{z}>0$. It is well known that $\tau$ can be extended continuously to $\arc$ and $\lim_{z\to\zeta}|\tau(z)|=1$, $\zeta\in\arc$.  If $\lambda\in \Omega$, let $\Lambda_\lambda:=\{z\in\Omega:|\tau(z)|<|\tau(\lambda)|\}$, $\partial\Lambda_\lambda:=\{z\in\Omega:|\tau(z)|=|\tau(\lambda)|\}$, $\overline{\Lambda}_\lambda^c:=\{z\in\Omega:|\tau(z)|>|\tau(\lambda)|\}$, when $\lambda\in \arc$, we consider $\partial\Lambda_{\lambda}:=\arc$ and $\overline{\Lambda}_\lambda^c:=\Omega$. Let $(\phi_n)$ denote a sequence of monic polynomials such that $\deg{\phi_n}=n$ for all $n$ and
\begin{equation}\label{limitOutArc2}
\lim_{n\to\infty}\frac{\phi_n(z)}{\capa{\arc}^n\tau^n(z)}=\mathfrak{F}(z),
\end{equation}
uniformly on compact subsets of $\Omega$, in which the analytic function $\mathfrak{F}$ has no zeros. There are several sequences of monic polynomials which satisfy condition \eqref{limitOutArc2} such as extremal polynomials with respect to a measure whose weight satisfies the Szeg\H{o} condition (see \cite{Wid67}). If $(\phi_n)$ is a sequence of monic orthogonal polynomials with respect to a measure $\mu$, $(I_{m,\lambda}(\phi_n))_{n\in\ZZp}$ is a sequence of polynomials orthogonal in a non-standard sense, i.e. they satisfy Sobolev-type orthogonality with respect to the inner product
$$
\langle f,g \rangle:=\sum_{k=0}^m\int f^{(k)}g^{(k)}\, d\mu_k,
$$
where $\mu_k=\delta_{\lambda},\, 0\le k\le m-1$, and $\mu_m=\mu$. There are several papers on this issue (e.g. \cite{AlPePiRe99,BePiMaUr11,PiBeUr10,PiRi17,Rod06,Sha17}). These polynomials are useful in Fourier analysis (\cite{DiMaPiUr19,Sha17}), numerical analysis (\cite{GaPePi00}), and so on (see \cite{PiRi17} and references therein).

\begin{theorem}\label{Th-IterAsymCompara} Let  $(\phi_n)$ be a sequence of monic polynomials which satisfies \eqref{limitOutArc2}.

\begin{enumerate}[(i)]
  \item If $\lambda \in \arc$,  then
\begin{equation*}\label{IterAsymCompara}
\lim_{n \to \infty} \frac{\phi_{n}(z)}{I_{m,\lambda}(\phi_n)(z)} = \psi^m(z),
\end{equation*}
uniformly on   compact subsets of $\Omega$, where the function $\psi(z):=\frac{\tau'(z)}{\tau(z)}$. Also, we have
$$
\wlim_{n\to\infty}\nu^*[I_{m,\lambda}(\phi_n)]=\mu_\arc,
$$
where $\mu_\arc$ is the equilibrium measure on the arc $\arc$ and
$$
\nu^*[I_{m,\lambda}(\phi_n)]=\sum_{\begin{subarray}{c}
I_{m,\lambda}(\phi_n)(\zeta)=0\\
\zeta\ne 0
\end{subarray}}\delta_\zeta.
$$

\item If  $\lambda \in \Omega$,  then
\begin{equation*}\label{Taylor-compara}
\lim_{n \to \infty} \frac{n^{m-1}\phi_n(z)}{I_{m,\lambda}(\phi_n)(z)} =  -\frac{(m-1)!\psi^m(z)}{(z-\lambda)^{m-1}}
\end{equation*}
uniformly on compact subsets of $\Lambda_\lambda$ and
\begin{equation*}\label{Taylor-compara2}
\lim_{n \to \infty} \frac{\phi_{n}(z)}{I_{m,\lambda}(\phi_n)(z)} = \psi^m(z),
\end{equation*}
uniformly on compact subsets in $\overline{\Lambda}_\lambda^c$. Moreover,  $\dsty
\wlim_{n\to\infty}\nu^*[I_{m,\lambda}(\phi_n)]=\mu_{\Lambda_\lambda},
$ where $\mu_{\Lambda_\lambda}$ is the equilibrium measure on the arc $\partial\Lambda_\lambda$.
\end{enumerate}
\end{theorem}

Section \ref{Sect-3} is devoted to the proof of the above theorem. Theorem \ref{TeoAsintIntUlt} is proved in Section \ref{SecUltrasph}. In Section \ref{Sect-2} we state some results about the location of the zeros of iterated integrals of polynomials and the next section includes the proof of Theorem \ref{TeoAsintPnn} as well as some of its consequences.

\section{Proof of Theorem \ref{TeoAsintPnn} and some consequences}\label{SectProofTh1}
Let $\dsty P_{n,m}(z):=\int_{0}^{1} (1-t)^{m-1} (1+zt)^n\;dt.$ The relations in \eqref{AsintQnn} are equivalent to prove that
\begin{equation*}
\label{AsintPnn}
P_{n,n}(z)=\left\{
\begin{array}{ll}
\frac{\sqrt{\pi}}{2^{2n}\sqrt{n}}\frac{(z+1)^{2n}}{z^n} (1+o(1))&\textit{if }z\in\mathcal{E}_1,\\ \\
\frac{1}{n(1-z)}(1+o(1))& \textit{if }z\in \mathcal{E}_2,
\end{array}
\right.
\end{equation*}
uniformly,  as $n\to\infty$, on compact subsets of each  stated domain.

Changing $t=\frac12(1-\frac1z)+\frac12(1+\frac1z)x$, a straightforward computation gives us
\begin{equation}
\label{forPnGn}
P_{n,n}(z):=\int_0^1(1-t)^{n-1}(1+zt)^n\, dt=\frac{(z+1)^{2n}}{2^{2n}z^{n}}\int_{w}^1(1-x^2)^{n-1}(1+x)\, dx,
\end{equation}
where $w=\frac{1-z}{1+z}$. Observe that the bi-linear function $w=\frac{1-z}{1+z}$ transforms $\overline{\DD}_1^c$ onto $\Re(w)<0$ and $1-w^2=\frac{4z}{(1+z)^2}$. Define
\begin{equation}
\label{polGnn}
G_{n}(w):=\int_{w}^1[(1-x^2)^{n}\, dx.
\end{equation}
We have $\dsty \int_{w}^1(1-x^2)^{n-1}(1+x)\, dx=G_{n-1}(w)+\frac{(1-w^{2})^n}{2n},$ so, the theorem is equivalent to check:
\begin{enumerate}[(i)]
\item Uniformly on compact subsets of $\mathcal{G}_2:=\{w\in \CC:|1-w^2|>1\}\cup \{w\in \CC:|1-w^2|<1$  and $\Re(w)>0\}$,
\begin{equation}
\label{AsintGnE2}
G_{n}(w)=\frac{(1-w^2)^{n+1}}{2nw}(1+o(1)).
\end{equation}
\item Uniformly on compact subsets of $\mathcal{G}_1:=\{w\in \CC:|1-w^2|<1\textit{ and }\Re(w)<0\}$,
\begin{equation}
\label{AsintGnE1}
G_{n}(w)=\frac{\sqrt{\pi}}{\sqrt{n}}(1+o(1)),
\end{equation}
\end{enumerate}
The integral \eqref{polGnn} does not depend on the curve of integration. Thus, we deform the integration interval $[w,1]$ as we need in each case.

First, we obtain \eqref{AsintGnE2}. Let $K\subset (\mathcal{G}_2\cap \{w\in\CC:\Re(w)\ge 0\})$ be a compact set and $w\in K$. From the change of variable $\frac{t^2}{n}=\frac{x^2-w^2}{1-w^2}$ and the dominated converge theorem, we get
\begin{align*}
\int_{w}^1(1-x^2)^n\,dx&=\frac{(1-w^2)^{n+1}}{n}\int_{0}^{\sqrt{n}}[(1-\frac{t^2}{n})^n\frac{t}{\sqrt{(1-w^2)\frac{t^2}{n}+w^2}}\,dt=\frac{(1-w^2)^{n+1}}{2nw}(1+o(1)).
\end{align*}
If $K\subset (\mathcal{G}_2\cap \{w\in\CC:\Re(w)\le 0\})$ is a compact set and $w\in K$, the proof is reduced to the  former case because
$\dsty G_n(w)=-G_n(-w)+2G_n(0),$ and
\begin{equation}
\label{Gn0}
G_n(0)=\frac{\Gamma(n+1)\Gamma(1/2)}{2\Gamma(n+3/2)}=\frac{\sqrt{\pi}}{2\sqrt{n}}(1+o(1)).
\end{equation}

Second, we get \eqref{AsintGnE1}. Let $K\subset \mathcal{G}_1$ and $w\in K$;
we can choose $w_1\in (-1,0)$ such that $K\subset \{u:\Re(u)<w_1\}$. Then
$$
G_n(w)=\int_w^{w_1}(1-x^2)^n\,dx+\int_{w_1}^0(1-x^2)^n\,dx+G_n(0)=2G_n(0)(1+o(1)).
$$
\hfill $\square$

The following result is a consequence of Theorem \ref{TeoAsintPnn}.
\begin{corollary}\label{CorAsintNroot}
We have
$$
\lim_n \left|Q_{n,n}(z)\right|^{1/n}=\left |\frac{(z+1)^2}{z}\right|,\quad z\in\mathcal{E}_1,
$$
$$
\lim_n \left|Q_{n,n}(z)\right|^{1/n}=4,\quad z\in\mathcal{E}_2,
$$
uniformly on compact subsets of each mentioned region. Moreover,
\begin{equation}
\label{AsintFrontera}
\lim_n \left(\max_{z\in\Gamma}\left|Q_{n,n}(z)\right|\right)^{1/n}=4.
\end{equation}
\end{corollary}

\begin{proof}
We will only check \eqref{AsintFrontera}. From \eqref{forPnGn}, it is equivalent to $\dsty
\lim_{n\to\infty}\left(\max_{w\in\partial \mathcal{G}_1}\left|G_n(w)\right|\right)^{1/n}=1. $

By the maximum principle, $\dsty G_n(0)\le \max_{w\in\partial \mathcal{G}_1}\left|G_n(w)\right|\le 1.$
From \eqref{Gn0} the proof is concluded straightforward.
\end{proof}

\begin{remark} If $z\in \overline{\mathcal{E}}_1:=\{z\in\CC:\left|\frac{(z+1)^2}{4z}\right|\ge 1,|z|\ge 1\}$, then there exists $K_1(z)>0$ ($K_1(z)$ only depends on $z$) such that
\begin{equation*}
\label{eqCompQnnGamma}
\frac{K_1(z)}{\sqrt{n}}\le |Q_{n,n}(z)|.
\end{equation*}
for $n$ large enough. In fact, from \eqref{polQnm} and \eqref{forPnGn}, the above relation is equivalent to
$$
\frac{K_2(w)}{\sqrt{n}}\le |G_n(w)|
$$
for $w\in \overline{\mathcal{G}}_1:=\{w\in \CC:|1-w^2|\le 1\textit{ and }\Re(w)\le0\}$, where $K_2(z)$ is a positive function.
The case $w=0$ is straightforward. If $w\ne 0$,  $\dsty G_n(w)=2G_n(0)-\int_{-w}^{1}(1-x^2)^n\, dx$ and there exists  $C(w)>0$ such that $\dsty \left|\int_{-w}^{1}(1-x^2)^n\, dx\right|\le \frac{C(w)}{n}.$

Actually, if $z$ is outside a disc with center at $1$ ($w$ is outside of a disc with center at $0$) with small radius, then we can take the positive constant $C(w)$ independent of $w$.

With the same arguments we obtain
\begin{equation}
\label{AsintQnn+1}
Q_{n,n+1}(z)=\left\{
\begin{array}{ll}
\frac{(z+1)^{2n+1}}{z^{n+1}} (1+o(1))&\textit{if }z\in\mathcal{E}_1,\\ \\
\frac{2^{2n+1}}{\sqrt{\pi n}(1-z)}(1+o(1))& \textit{if }z\in \mathcal{E}_2,
\end{array}
\right.
\end{equation}
as $n\to\infty$ uniformly on compact subsets of the mentioned regions.

\end{remark}
\subsection{Proof of Corollary \ref{CorZeroDistQ}}

Given that  $w=\frac{(z+1)^2}{z}$ is a conformal transform from $\mathcal{E}_1$ onto $\overline{\DD}_4^c$, and $\capa{\DD_4})=4$ (see \cite[Th. 5.2.3]{Ran95}), we have
$ \dsty \capa{\Gamma}=4.$ Moreover, by the subordination principle (\cite[Th. 4.3.8]{Ran95}, for each Borel set $B\subset \CC$,
$\mu_\Gamma(B)=m(\varphi(B\cap \Gamma)),$ where $dm=d\theta/(2\pi)$ on $\partial\DD_1$.

Let $\mu$ be a weak-$*$ limit of $\nu[Q_{n,n}]:=\frac{1}{n}\sum_{z:Q_{n,n}(z)=0}\delta_{z}$, i.e. there exists a subsequence $(\nu[Q_{n_k,n_k}])$ such that
$\lim_{k\to\infty}\int f(\zeta)\, d\nu[Q_{n_k,n_k}](\zeta)=\int f(\zeta)\, d\mu(\zeta)$
for all continuous function $f$ in $\CC$ with compact support.  To simplify  notation, we write $n$ instead of $n_k$. By Theorem \ref{TeoAsintPnn}, $\supp{\mu}\subset \Gamma$. It is well known that\footnote{Hereafter, if $\nu$ is a positive Borel measure with compact support in the complex plane, its logarithm potential and energy are respectively $V(\nu,z):=\int\frac{1}{|z-x|}\, d\nu(x)$ and $I(\nu):=\int V(\nu,z)\, d\nu(z)$.}
$$
\lim_{n\to\infty}V(\nu[Q_{n,n}],z)=V(\mu,z)
$$
uniformly on compact subsets of $\CC\setminus \Gamma$. Then, by \eqref{polQnm}, and Corollary
\ref{CorAsintNroot}, we obtain
$$
V\left(\nu[Q_{n,n}],z\right)=\frac{1}{n}\log\left(\frac{(n-1)!n!}{(2n)!}\right)+V\left(\nu[P_{n,n}],z\right)
$$
and
$$
V(\mu,z)=\log\left(\frac14 \right)+
\left\{
\begin{array}{ll}
0&\textit{ if }z\in \mathcal{E}_2,\\ \\ \log\left|\frac{4z}{(z+1)^2}\right| &\textit{ if }z\in \mathcal{E}_1.
\end{array}
\right.
$$

Then, from the identity principle for harmonic functions (see \cite[Th. 1.1.7]{Ran95}) we get  $\supp{\mu_{\Gamma}}=\Gamma$ and from the principle of descent, if $z\in \Gamma$ and $(z_n)\subset  \mathcal{E}_2$ such that $\dsty \lim_{n\to\infty}z_n=z$,
$$
V(\mu,z)\le \liminf_{n\to\infty}V(\nu[Q_{n,n}],z_n)=\log\left(\frac14 \right),
$$
Since $\mu$ is a probability measure, $I(\mu)=\int V(\mu,z)\, d\mu(z)\le \log(\frac14)$, but $\mu_\Gamma$ is the unique measure which minimizes the energy between the probability measures with support on $\Gamma$, and $I(\mu_\Gamma)=\log\frac1{\capa{\Gamma}}=\log(\frac14)$. Therefore, $\mu=\mu_\Gamma$, $\dsty
\wlim_{n\to\infty}\nu[Q_{n,n}]=\mu_{\Gamma}.$ \hfill  $\square$

In Figures \ref{GrapQ} we can see the zeros of $Q_{n,m}$ for several values of $n$ and $m$.

\begin{figure}[th]\centering
  \includegraphics[width=.8\textwidth]{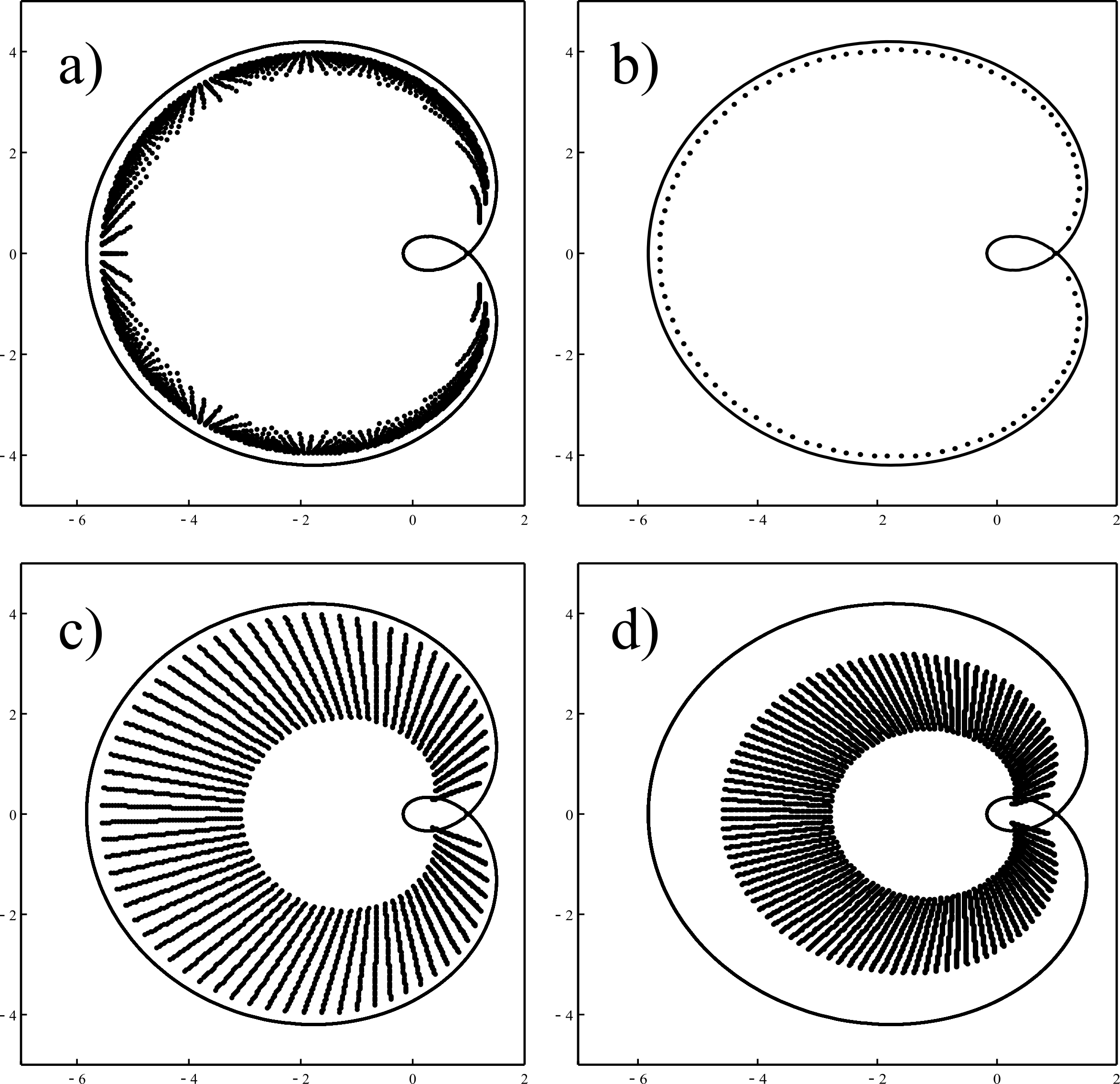}
  \vspace*{8pt}
  \caption{a) Zeros of $Q_{n,n}$ for $n=20,\ldots,70$. b) Zeros of $Q_{100,100}$. c) Zeros of $Q_{n,m}$ for $n=70$ and $m=20,\ldots,70$. d) Zeros of $Q_{n,m}$ for $n=100$ and $m=20,\ldots,70$.  In all cases the curve is $\left|\frac{(z+1)^2}{4z}\right|=1$.}
  \label{GrapQ}
\end{figure}

\section{Asymptotic analysis for the integral of ultraspherical polynomials}\label{SecUltrasph}

\textit{Proof of Theorem \ref{TeoAsintIntUlt}.}

We shall do induction on the parameter $\alpha$. For $\alpha=0$, we have Legendre polynomials. By Rodrigues' formula (see \cite[(4.3.1)]{Sze75}) the monic Legendre polynomials $\widehat{P}_n^{(1/2)}$ are given by
$$
\widehat{P}_n^{(1/2)}(z):=(-1)^n\frac{n!}{(2n)!}\partial_z^n\left((1-z^2)^n\right).
$$
So, $\dsty
I_{2n}(\widehat{P}_{2n}^{(1/2)})(z)=(-z^2)^nQ_{n,n}(-z^2)$ and $\dsty I_{2n+1}(\widehat{P}_{2n+1}^{(1/2)})(z)=-(-z^2)^{n+1}Q_{n,n+1}(-z^2).$

From \eqref{AsintQnn+1} , Theorem \ref{TeoAsintPnn}, and Corollary \ref{CorZeroDistQ}, we obtain
\footnote{$\mathcal{F}_1:=\{z\in\CC:\left|\frac{z^2-1}{2z}\right|>1,\,|z|>1\}$, $\mathcal{F}_2:=(\{z\in\CC:\left|\frac{z^2-1}{2z}\right|<1\})\cup (\{z\in\CC:\left|\frac{z^2-1}{2z}\right|>1,\,|z|<1\}$. Observe that the coefficient of $z^{2n}$ in $I_{2n}(\widehat{P}_{2n}^{(1/2)})$ is $\frac{P_{2n}^{(1/2)}(0)}{\kappa_{2n}}\sim \frac{(-1)^n2^{2n}}{\sqrt{\pi n}}$ as $n\to\infty$, where $\kappa_{2n}$ is the leading coefficient of $P_{2n}^{(1/2)}$.}
\begin{equation*}
\label{AsintI2n}
I_{2n}(\widehat{P}_{2n}^{(1/2)})(z)=\left\{
\begin{array}{ll}
(z^2-1)^{2n} (1+o(1))&\textit{if }z\in\mathcal{F}_1,\\ \\
\frac{(-1)^n2^{2n}}{\sqrt{\pi n}}\frac{z^{2n}}{1+z^2}(1+o(1))& \textit{if }z\in \mathcal{F}_2,
\end{array}
\right.
\end{equation*}
\begin{equation*}
\label{AsintI2n+1}
I_{2n+1}(\widehat{P}_{2n+1}^{(1/2)})(z)=\left\{
\begin{array}{ll}
(z^2-1)^{2n+1} (1+o(1))&\textit{if }z\in\mathcal{F}_1,\\ \\
\frac{(-1)^{n}2^{2n+1}}{\sqrt{\pi n}}\frac{z^{2(n+1)}}{1+z^2}(1+o(1))& \textit{if }z\in \mathcal{F}_2,
\end{array}
\right.
\end{equation*}
as $n\to\infty$, where $\mathcal{F}_j$ are the pre-image of $\mathcal{D}_j$ under the transformation $\varphi_1(z)=-z^2$. Moreover, its zero distribution\footnote{The zeros of $I_n(\widehat{P}_n^{(1/2)})$ different from $0$.} $\beta$ is supported on $\varphi_1(\Gamma)$ and for each Borel set $E$ this measure satisfies $\dsty \beta(E)=\mu_{\Gamma}(\varphi_1(E)).$

Next, we assume that the statement holds for $\alpha$.  According to \cite[(4.21.7)]{Sze75}, we have
$$
\partial_z\left(P_{n}^{(\alpha+1/2)}(z)\right)=\frac12(n+2\alpha+1)P_{n-1}^{(\alpha+3/2)}(z),
$$
so,\footnote{Remember that $I_m(z)$ are monic polynomials and $\partial^j_z(I_m(z))=0,\, j=0,1,\ldots,m-1$. The factor $\frac{(2n-2)!}{(n-2)!n!}$ is to guarantee that $\partial^{n-2}_z(I_{n-1}\left(\widehat{P}_{n-1}^{(\alpha+3/2)})\right)=0.$} $$
I_{n-1}\left(\widehat{P}_{n-1}^{(\alpha+3/2)}\right)(z)\\=\left(I_{n-2}\left(\widehat{P}_{n}^{(\alpha+1/2)}\right)(z)-\widehat{P}_{n}^{(\alpha+1/2)}(0)\frac{(2n-2)!}{(n-2)!n!}z^{n-2}\right).
$$
Moreover, it holds $\dsty  I_{n-2}\left(\widehat{P}_{n}^{(\alpha+1/2)}\right)(z)=\frac{1}{2n(2n-1)}\partial_z^2\left(I_{n}\left(\widehat{P}_{n}^{(\alpha+1/2)}\right)(z)\right).$

Next, we consider $n=2k+2$ even. By the Cauchy integral formula for the derivative and induction hypothesis, we get
\begin{align*}
I_{2k}\left(\widehat{P}_{2k+2}^{(\alpha+1/2)}\right)(z)= &\frac{1}{(4k+2)(4k+1)}\partial_z^2\left(I_{2k+2}\left(\widehat{P}_{2k}^{(\alpha+1/2)}\right)(z)\right)\\
  = &  (z^2-1)^{2k-\alpha}z^{2\alpha} (1+o(1))
\end{align*}
uniformly on compact subsets of $\mathcal{F}_1$.
 Observe that on $\mathcal{F}_1$ $\dsty \left|\frac{z^2-1}{2z}\right|>1,$ and by \cite[(4.7.31)]{Sze75}
$$
\widehat{P}_{2k+2}^{(\alpha+1/2)}(0)=\frac{(-1)^{k+1}\Gamma(k+\alpha+3/2)\Gamma(2k+3)}{ 2^{2k+2}\Gamma(k+2)\Gamma(2k+\alpha+5/2)}.
$$
Moreover, by Stirling's formula
$$
\frac{\Gamma(k+\alpha+3/2)}{\Gamma(k+2)}=k^{\alpha-1/2}(1+o(1)),\quad \frac{(4k+2)!}{(2k)!(2k+2)!}=\frac{2^{4k+2}}{\sqrt{2\pi k}}(1+o(1)).
$$
Thus,
\begin{align*}
  I_{2k}(\widehat{P}_{2k+2}^{(\alpha+1/2)})(z)-\widehat{P}_{2k+2}^{(\alpha+1/2)}(0)\frac{(2k)!}{k!(k+2)!}z^{2k-2} = & I_{2k}(\widehat{P}_{2k+2}^{(\alpha/2)})(z)(1+o(1)) \\
  = & (z^2-1)^{2k+2-\alpha} z^{2\alpha}(1+o(1)).
\end{align*}
uniformly on compact subsets of $\mathcal{F}_1$.
Therefore,
$$
I_{2k+1}\left(\widehat{P}_{2k+1}^{(\alpha+3/2)}\right)(z)=(z^2-1)^{2k+1-(\alpha+1)} z^{2(\alpha+1)}(1+o(1)),
$$
uniformly on compact subsets of $\mathcal{F}_1$.

Consider $z\in\mathcal{F}_2$, we have
$$
\frac{1}{(4k+4)(4k+3)}\partial_z^2\left(I_{2k+2}\left(\widehat{P}_{2k+2}^{(\alpha+1/2)}\right)(z)\right)=
\frac{(-1)^{k+1}2^{2k+2-\alpha}z^{2k}}{2^2\sqrt{\pi k}(1+z^2)}(1+o(1)).
$$
In fact, it is immediately checked that
\begin{align*}
  \partial_z^2\left(\frac{(-1)^{k+1}2^{2k+2-\alpha}z^{2k+2}}{\sqrt{\pi k}(1+z^2)}\right) = & \frac{(-1)^{k+1}2^{2k+2-\alpha}}{\sqrt{\pi k}} \partial_z^2\left(\frac{z^{2k+2}}{1+z^2}\right)  \\
  = & \frac{(-1)^{k+1}2^{2k+2-\alpha}}{\sqrt{\pi k}} \\
  &  \times \frac{2 z^{2 n} (1 - 3 z^2 + n (1 + z^2) (3 + 2 n + (-1 + 2 n) z^2))}{(1 +  z^2)^3},
\end{align*}
and $\dsty \frac{z^{2k}}{1+z^2}-z^{2k}=\frac{-z^{2k+2}}{1+z^2}.$ Therefore,
$$
I_{2k+1}\left(\widehat{P}_{2k+1}^{(\alpha+3/2)}\right)(z)=\frac{(-1)^{k}2^{2k+1-(\alpha+1)}}{\sqrt{\pi k}}\frac{z^{2k+2}}{1+z^2}(1+o(1))
$$
uniformly on compact subsets of $\mathcal{F}_2$.

If $n$ is odd, then $\widehat{P}_n^{(\alpha+1/2)}(0)=0$,
$$
I_{2n}\left(\widehat{P}_{2n}^{(\alpha+3/2)}\right)(z)=\frac{1}{(4n+2)(4n+1)}\partial_z^2\left(I_{2n+1}\left(\widehat{P}_{2n+1}^{(\alpha+1/2)}\right)(z)\right),
$$
and the proof is concluded as the former case.

The conclusion about the limiting distribution of the zero of $I_n\left(\widehat{P}_n^{(\alpha+1/2)}\right)$  follows straightforward from their asymptotic behavior and the unicity theorem for potentials (see \cite[Theorem 2.1, p. 97]{SafTot97}).

\hfill $\square$

In Figures \ref{ZerosIntLegendre} and \ref{ZerosIntUltra} we can see the zeros of iterated integral of ultraspherical polynomials, in particular, of Legendre polynomials, for different values of $n$.

\begin{figure}[th]\centering
  \includegraphics[width=\textwidth]{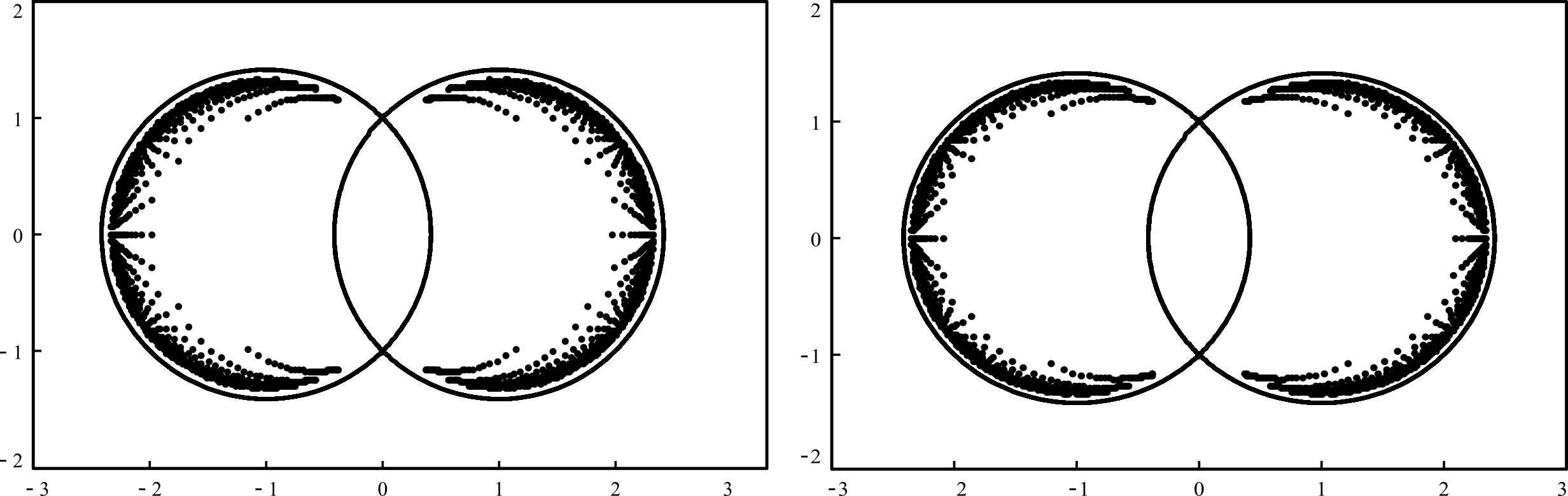}
  \vspace*{8pt}
  \caption{Zeros of the iterated integral of Legendre polynomials $I_n(P_n^{(1/2)})$: to the left, for $n$ even $n=10$ to $n=80$. To the right, for $n$ odd $n=11$ to $n=81$. In both the curve $\left|\frac{1-z^2}{2z}\right|=1$.}
  \label{ZerosIntLegendre}
\end{figure}

\begin{figure}[th]\centering
  \includegraphics[width=\textwidth]{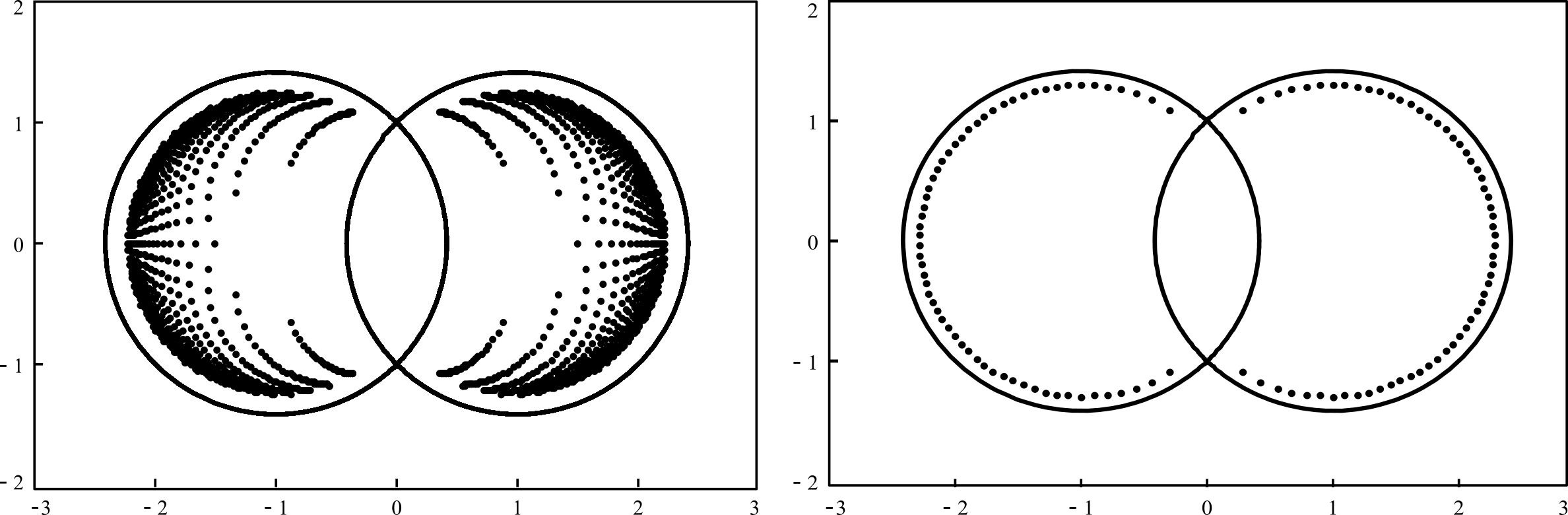}
  \vspace*{8pt}
  \caption{Zeros of the integral of ultraspherical polynomials $I_n(P_n^{(13/2)})$: to the left, for $n$ even, $n=20,\ldots,160$. To the right, for $n=120$. In both the curve $\left|\frac{1-z^2}{2z}\right|=1$.}
  \label{ZerosIntUltra}
\end{figure}

\section{Zeros of polynomials with critical points on a disc}\label{Sect-2}

It is clear  that if we know the location of the critical points   of a polynomial and one of its  zeros, the remaining zeros are uniquely determined. Nonetheless, there are only a few  general results about  the  zero locations of polynomials in terms of theirs critical points and a given zero, most of them contained in \cite[\S 4.5]{RahSch02} and \cite{BoChDi95}.

Since
$$
Q_{n,m}(z)=\frac{(n+m)!}{m!}\frac1{z^{m}}I_m[(1+z)^n]=\frac{1}{z^m}\left((1+z)^{n+m}-\sum_{k=0}^{m-1}\binom{n+m}{k}z^k\right), \quad \text{we have,}
$$

\begin{lemma}
$$
Q_{n,m}(z)=
\left\{
\begin{array}{ll}
\frac{(z+1)^{n+m}}{n^{m-1}z^m}(1+o(1))&\text{if  }|z+1|>1,\\
\\
-\frac{1}{z(m-1)!}(1+o(1))&\text{if  }|z+1|<1.
\end{array}
\right.
$$
as $n\to\infty$. Moreover, $\dsty \wlim_{n\to\infty}\nu[Q_{n,m}]=\mu_{\DD_{1}(-1)},$ where $\mu_{\DD_{1}(-1)}$ is the equilibrium measure on $\DD_{1}(-1)$. So it is given by the normalized arc length on $\DD_{1}(-1)$.
\end{lemma}

If $p_n$ is a polynomial of degree $n$, $\dsty p_n(z)= \sum_{k=0}^n\binom{n}{k}a_kz^k, $
$$
 I_{m,0}(p_n)(z)=  \frac{(n+m)!}{n!}\sum_{k=0}^n\binom{n}{k}\frac{k!}{(m+k)!}a_kz^{k+m}=z^m\sum_{k=0}^n\binom{n}{k}c_kz^k,
$$
then $c_k=a_kb_k$, where $\dsty \sum_{k=0}^n\binom{n}{k}b_kz^k=Q_{n,m}(z).$

This  means that, as Borwein, Chen and Dilcher (\cite{BoChDi95}) observed, $I_{m,0}$ is a Hadamard product of $p_n$ and $Q_{n,m}$. Then, by a theorem of Szeg\H{o} and Schur \cite{BoChDi95}, for $n$ large enough, we have:

\begin{corollary}
If the zeros of $p_n$ lie in the disc $\DD_r$, then the zeros of $I_{m,0}(p_n)$ lie in $\DD_{(2+\epsilon(n))r}$, where $(\epsilon(n))$ is a deacreasing sequence of positive number with $\dsty \lim_{n\to\infty}\epsilon_n=0$.
\end{corollary}

The next result also helps to locate the zeros of the   iterated integral of a polynomial. This is an extension of  \cite[Th. 5.7.8]{She02} and  the proof is carried out with analogous  arguments.

\begin{lemma}\label{Zeros_Transf} Let $P$ be a polynomial of degree $n \geq 2$ with  all its critical points in  the closed disc $\overline{\DD}_r$, where $r\in \RRp$ is fixed. If  $P(\lambda)=P(z)=0$, with $\lambda,z \in \CC$, then
\begin{enumerate}[(i)]
   \item there exists $w\in \overline{\DD}_1 $ such that
   \begin{equation}\label{Zeros_Transf-equ}
   z= F_r(w):=2r w - \overline{\lambda}w^{2}.
\end{equation}
 \item \label{ii}$||z|-|\lambda|| \leq {2r}$.
   \item   $F_r$ is univalent  on $\overline{\DD}_1$ if  and only if $|\lambda|\leq r$.
\end{enumerate}
\end{lemma}

\begin{proof} As $z$ and $\lambda$ are zeros of $P$, from the bisector lemma (see \cite[Th. 4.3.1]{RahSch02}),  if we  draw a straight line $\ell$ which cuts perpendicularly the segment joining the two zeros at its middle point, then $P^{\prime}$ has at least one zero in each of the closed half planes in which $\ell$ divides the complex plane. But, we have assumed that all the zeros of $P^{\prime}$   lie in $\overline{\DD}_r$ and therefore $\ell$ must intersect $\partial \DD_r $. Hence, there exists $u \in  \partial\DD_1$ such that $|z-r\,u|=|r\,u-\lambda|=  |r-{\overline{\lambda}}\,{u}|$.
 It follows that there exists $v \in \partial\DD_1$ such that $\dsty z-r\,u= v\, (r-{\overline{\lambda}}\,{u})$, where we have
$z=r(u+ v)- \overline{\lambda}\,u\,v$. This expresses $z$ as a value of a symmetric linear form in the variables $u$ and $v$ taking their values on $\partial\DD_1$, and therefore in $\overline{\DD}_1$. It follows from Walsh's coincidence lemma \cite[Th. 3.4.1b]{RahSch02}  that $z$ is a value of the polynomial obtained by putting $w=u=v$ with $w\in \overline{\DD}_1$, which establishes \eqref{Zeros_Transf-equ} and the inequality in statement \ref{ii} as an immediate consequence.

 If $\lambda=0$ then obviously $F_r$ is univalent. Assume that  $\lambda\neq0$,  if there exist $w_1,w_2 \in \overline{\DD}_1$ such that $w_1\neq w_2$ and $F_r(w_1)=F_r(w_2)$, we get that $w_1+w_2= 2(r /{\overline{\lambda}})$.       Therefore,  $F_r$ is univalent on $\overline{\DD}_1$ if  and only if $|\lambda|\leq r$   and  we get the third statement of the theorem.
 \end{proof}

\begin{remark}Under the above assumptions,  as a consequence of Lemma \ref{Zeros_Transf}, the possible region of zeros of $P$ is the set $F_r(\overline{\DD}_1)$ and  if $|\lambda|\leq r$  then $F_r$  maps $\partial\DD_1$ onto a Jordan curve (for $r=1$ see Figure \ref{ZerosRegion}).
\end{remark}

\begin{figure}[th]\centering
  \includegraphics[width=\textwidth]{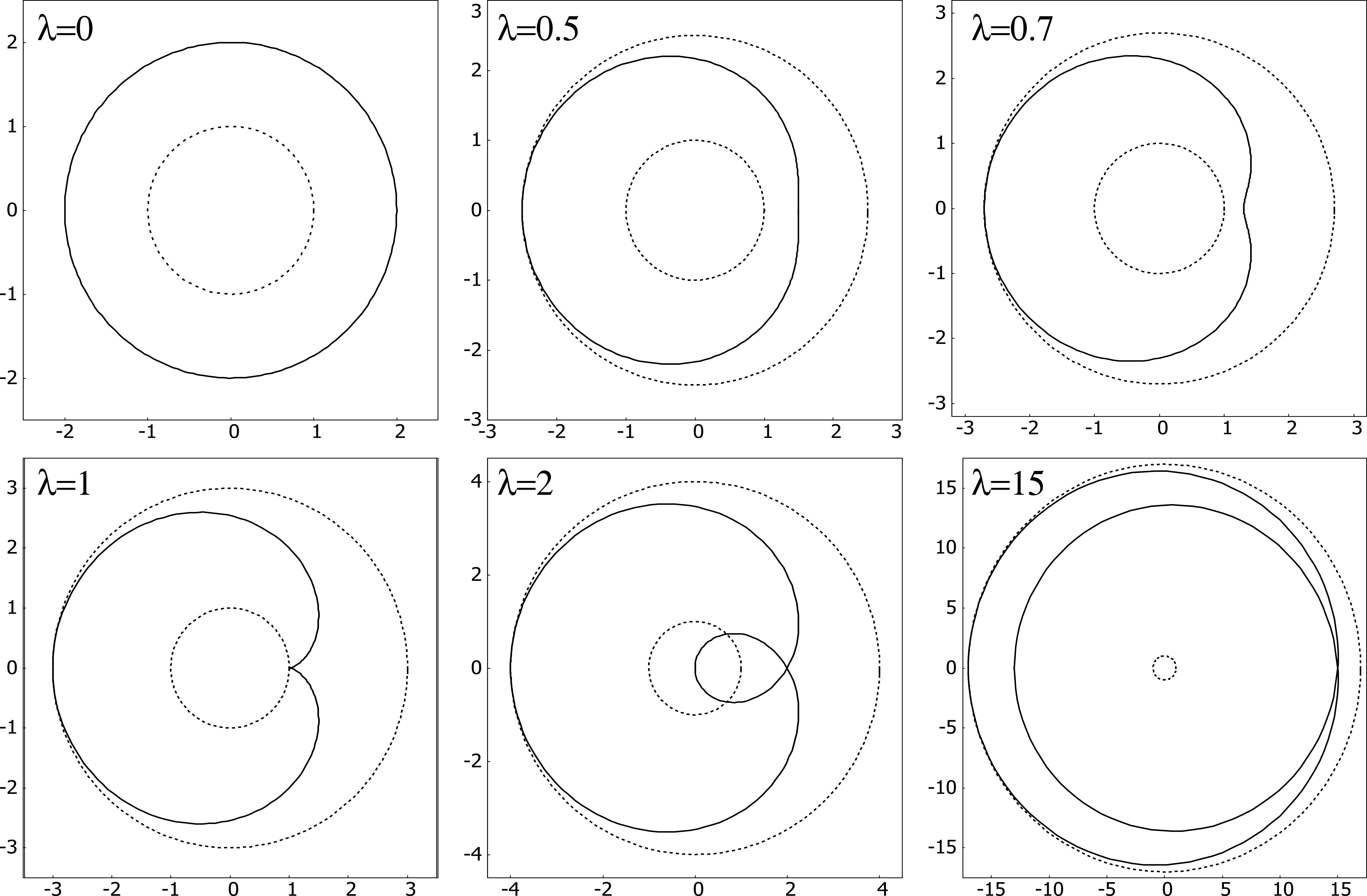}
  \vspace*{8pt}
  \caption{The cardioidal curve  $F_r(\partial\DD_1)$, for $r=1$ and several values of $\lambda$. The interior circle is $\partial\DD_1$ and the exterior one is given by $|z|={2+ |\lambda|}$.}
  \label{ZerosRegion}
\end{figure}

\begin{corollary}  Given two integers  $n,m>0$ and $\lambda \in \CC$, let $\rho=2^m(|\lambda|+r)-|\lambda|$. If all the zeros of $\phi_n$ are in $\overline{\DD}_r$, then all the zeros of the polynomials $I_{m,\lambda}(\phi_n)$  lie in the  closed  disc $\overline{\DD}_{\rho}$.
\end{corollary}
\begin{proof}For $m=1$, as all the zeros of $\phi_n$   lie  in $\DD_r$  (i.e. the critical points of $I_{1,\lambda}(\phi_n)$), the assertion follows from Lemma \ref{Zeros_Transf}. The rest of the proof runs by induction.
\end{proof}

\section{Proof of Theorem \ref{Th-IterAsymCompara}}\label{Sect-3}

The proof of Theorem \ref{Th-IterAsymCompara} is divided into three subsections: first, two auxiliary lemmas, second, when $\lambda\in\arc$, and  third, when $\lambda\not\in\arc$.

\subsection{Auxiliary lemmas}
The following result plays a main role in  obtaining  the  strong asymptotic behavior of the $m$th iterated integrals $I_{m,\lambda}(\phi_n)$.

\begin{lemma} \label{lemaLimPolInt} Set $\lambda\in\Omega$, and $K_1$, $K_2$ two compact sets with $K_1\subset (\Lambda_\lambda\cup \arc)$ and $K_2\subset \overline{\Lambda}_\lambda^c$. Let $z_1\in K_1,z_2\in K_2$ and $\{\phi_n\}_{n\in\ZZp}$ be a sequence of polynomials which satisfies \eqref{limitOutArc2}. Then
\begin{equation}\label{eqLemmaLimInt}
\lim_{n\to\infty}\left|\int_{z_1}^{z_2}\phi_n(s)\, ds\right|^{1/n}=\capa{\arc}|\tau(z_2)| \quad \text{uniformly on $K_1$ and $K_2$.}
\end{equation}
\end{lemma}

\begin{proof}
For $z_1\in K_1$ and $z_2\in K_2$, denote $J_n:=\int_{z_1}^{z_2}\phi_n(s)\, ds$ .
This integral is independent of the contour of integration from $z_1$ to $z_2$.  Then,  by the maximum principle for holomorphic functions and \eqref{limitOutArc2} we have
$$
\limsup_{n\to\infty}|J_n|^{1/n}\le \left(\int_{z_1}^{z_2}\, |ds|\max_{\{z\in\CC:|\tau(z)|= |\tau(z_2)| \}}|\phi_n(z)|\right)^{1/n}=\capa{\arc}|\tau(z_2)|.
$$
Since the convergence in \eqref{limitOutArc2} is uniform on compact subsets of $\Omega$, the  above relation also holds  uniformly on $K_1$ and $K_2$.

The proof of the inequality
\begin{equation}\label{eqLowerBoundOnGamma2_F}
\liminf_{n\to\infty} \left|J_n \right|^{1/n}\ge \capa{\arc}|\tau(z_2)|
\end{equation}
requires a more detailed analysis. We chose $\zeta_0$ near to $z_2$ such that $[\zeta_0,z_2]\cap \arc=\emptyset$,
\begin{equation}\label{eqSelLambda0}
|\tau(\zeta_0)|>|\tau(\lambda)|\text{ and }\text{Arg}(\tau(\zeta_0))=\text{Arg}(\tau(z_2)).
\end{equation}
Let $\Delta_1$ be a Jordan rectifiable arc from $z_1$ to $\zeta_0$ in $\{z\in\CC:|\tau(z_1)|\le |\tau(z)|\le|\tau(\zeta_0)|\}$ and let $\Delta_2$ be  the arc given by $r(t)=\tau^{-1}(t\tau(z_2)+(1-t)\tau(\zeta_0))$, $t\in[0,1]$, which satisfies
$$
\text{Arg}(\tau(r(t)))\text{ is constant and }r'(t)=\frac{\tau(z_2)-\tau(\zeta_0)}{\tau'(t\tau(z_2)+(1-t)\tau(\zeta_0))}\ne 0,\quad t\in[0,1].
$$
We take as integration path in $J_n$ the curve $\Delta:=\Delta_1+\Delta_2$. Then,
\begin{equation*}\label{eqCotInfInt}
|J_n|= \left|\int_{\Delta}\phi_n(s)\, ds \right|\ge  \left|\int_{\Delta_2}\phi_n(s)\, ds \right|-\int_{\Delta_1}|\phi_n(s)|\, |ds|.
\end{equation*}
By the maximum modulus principle and \eqref{limitOutArc2},  it follows
\begin{equation}\label{eqUpBoundGamma1}
\limsup_{n}\left(\int_{\Delta_1}|\phi_n(s)|\, |ds|\right)^{1/n}\le \capa{\arc}|\tau(\zeta_0)|.
\end{equation}
Since  $\mathfrak{F}(r(1))\overline{\tau'(r(1))}\ne 0$, we have that its real or its imaginary part is different from zero. We assume that $\dsty
\Re\left(\mathfrak{F}(r(1))\overline{\tau'(r(1))}\right)\ne 0.$ Other case is reduced to this one by multiplying $\phi_n$ by $i$. So, by \eqref{limitOutArc2}, we can chose $\Delta_2^1\subset \Delta_2$ a piece of arc of $\Delta_2$ containing $z_2$ such that
$$
\lim_{n\to\infty} \Re\left(\frac{\phi_n(s)\overline{\tau'(s)}}{\capa{\arc}^n\tau^n(s)|\tau'(s)|^2}\right)= \Re\left(\frac{\mathfrak{F}(s)\overline{\tau'(s)}}{|\tau'(s)|^2}\right)\ne 0
$$
uniformly on $\Delta_2^1$ and
$H_n(s):=\Re\left(\frac{\phi_n(s)\overline{\tau'(s)}}{\capa{\arc}^n\tau^n(s)|\tau'(s)|^2}\right)$
has constant sign for $s\in\Delta_2^1$ and $n$ large enough. From \eqref{eqSelLambda0},  we have
\begin{align} \label{eqLowerBoundOnGamma2}
\left|\int_{\Delta_2}\phi_n(s)\, ds \right|\ge  &  \left|\int_{\Delta_2^1}\frac{\phi_n(s)}{(\capa{\arc})^n\tau^n(s)}(\capa{\arc})^n|\tau^n(s)|\, ds\right| -\int_{\Delta_2\setminus\Delta_2^1}|\phi_n(s)|\, |ds| \\  \nonumber  \ge & \Big|\int_{J}H_n(r(t))  (\capa{\arc})^n|\tau^n(r(t))|\, dt\Big| -\int_{\Delta_2\setminus\Delta_2^1}|\phi_n(s)|\, |ds|
\end{align}
where $J$ denote the parameter interval for the $\Delta_2^1$ path. By the chose of $\Gamma_2^1$, it follows
\begin{equation*}\label{eqBoundExtMainArc}
\limsup_{n\to\infty}\left(\int_{\Delta_2\setminus\Delta_2^1}|\phi_n(s)|\, |ds|\right)^{1/n}<|\capa{\arc}\tau(z_2)|.
\end{equation*}
Now,  we deduce
\begin{equation}\label{eqBoundMainArc}
\lim_{n\to\infty} \left|\int_{J}H_n(r(t))(\capa{\arc})^n|\tau^n(r(t))|\, dt\right|^{1/n}=|\capa{\arc}\tau(z_2)|.
\end{equation}
In fact, since $|\tau(r(t))|\le |\tau(r(1))|=|\ell(z_2)|$ for all $t\in J$, we have
$$
\limsup_{n\to\infty} \left|\int_{J}H_n(r(t))(\capa{\arc})^n|\tau^n(r(t))|\, dt\right|^{1/n}\le|\capa{\arc}\tau(z_2)|.
$$
On the other hand, given $\epsilon\in(0,|\tau(z_2)|)$,  we chose $\delta>0$ such that $[1-\delta,1]\subset J$ and $|\tau(r(t))|>|\tau(z_2)|-\epsilon,\, \forall t\in J$. Then,
$$
\left|\int_{J}H_n(r(t))(\capa{\arc})^n|\tau^n(r(t))|\, dt\right| \ge \left|\int_{[1-\delta,1]}H_n(r(t))\, dt\right|(\capa{\arc})^n(|\tau(z_2)|-\epsilon)^n.
$$
and we get $\dsty \liminf_{n\to\infty} \left|\int_{J}H_n(r(t))(\capa{\arc})^n|\tau^n(r(t))|\, dt\right|^{1/n}\ge|\capa{\arc}\tau(z_2)|.$

Once we have obtained \eqref{eqBoundMainArc}, from \eqref{eqSelLambda0}--\eqref{eqUpBoundGamma1} and \eqref{eqLowerBoundOnGamma2}--\eqref{eqBoundMainArc} follow readily \eqref{eqLemmaLimInt} and \eqref{eqLowerBoundOnGamma2_F}.
Observe that the limit in \eqref{eqLowerBoundOnGamma2_F} is uniform on $z_1$ and $z_2$. By the continuity of $\frac{\mathfrak{F}(s)\overline{\tau'(s)}}{|\tau'(s)|^2}$, we can consider that the relations \eqref{eqSelLambda0}--\eqref{eqBoundMainArc} hold for all $z_1'$ and $z_2'$ in a closed disc around $z_1$ and $z_2$.
\end{proof}

The following lemma is a well-known result, but we include a proof for an easy reading.

\begin{lemma}\label{Lemma_MarkovFuntion} We have
\begin{equation*}\label{MarkovFunEquilibrium}
    \psi(z)=\int \,\frac{d\mu_{\arc}(w)}{z-w},\quad z\in \Omega, \quad \text{where, as it was defined in Theorem \ref{Th-IterAsymCompara}, $\psi(z)=\frac{\tau'(z)}{\tau(z)}$.}
\end{equation*}

\end{lemma}

\begin{proof}
Let $\{\phi_n\}$ be a sequence of monic polynomials, which satisfy \eqref{limitOutArc2}. So, theirs zeros tend to $\arc$ and
$\dsty
\wlim_{n\to\infty}\nu[\phi_n]=\mu_\arc.
$  (See arguments to check $\wlim_{n\to\infty}\nu[I_{1,\lambda}(\phi_n)]=\mu_\arc$ in the proof of statement (i) of Theorem \ref{Th-IterAsymCompara} in the next section.) Let $(g_n)$ be the sequence of analytic functions $$g_n(z):=(\capa{\arc})^n\mathfrak{F}(z)(\tau(z))^n, \quad z\in\Omega, \quad n\in\NN.$$ Then, $\lim_{n\to\infty}\frac{g_n'(z)}{ng_n(z)}=\psi(z),
$
uniformly on compact subsets of $\Omega$. Since $$\dsty
\left(\frac{\phi_n(z)}{g_n(z)}\right)'=\frac{\phi_n(z)}{g_n(z)}\left(\frac{\phi_n'(z)}{\phi_n(z)}-\frac{g_n'(z)}{g_n(z)}\right),
$$  we get $\dsty \lim_{n\to\infty}\frac{\phi_n'(z)}{n\phi_n(z)}=\lim_{n\to\infty}\frac{g_n'(z)}{ng_n(z)}=\psi(z),$ uniformly on compact subsets of $\Omega$ and
$$
\psi(z)=\lim_{n\to\infty}\frac{\phi_n'(z)}{n\phi_n(z)}=\lim_{n\to\infty}\int\frac{d\nu[\phi_n](w)}{z-w}=\int\frac{d\mu_\arc(w)}{z-w},\quad z\in \Omega.
$$
\end{proof}

\subsection{Proof of Theorem \ref{Th-IterAsymCompara}(i): $\lambda \in \arc$}

We only consider the case  $m=1$ because another step of induction on $m$ follows with the same argument. It is well known that,  if $P_n$ is a monic polynomial of degree $n\ge 1$, we have
$\Vert P_n\Vert_\arc\ge (\capa{E})^n.
$ (See \cite[Th.  5.5.4]{Ran95}.)
Combining this relation with the definition of $I_{1,\lambda}(\phi_{n})$ and the condition \eqref{limitOutArc2} on $\phi_n$, it  plainly follows that $\dsty
\lim_{n\to\infty}\|I_{1,\lambda}(\phi_{n})\|_\arc^{1/n}=\capa{\arc}. $
Thus, since $\arc$ has empty interior and a connected complement (see \cite{BlSaSi88}), we have
\begin{equation}
\label{WeakLimit}
\wlim_{n\to\infty}\nu[I_{1,\lambda}(\phi_n)]=\mu_{\arc}.
\end{equation}
As $\tau$ has no zeros in $\Omega$, by Lemma \ref{lemaLimPolInt}, we know  that the zeros of $(I_{1,\lambda}(\phi_n))$  converge to $\arc$. So, we obtain
$$
\frac{\phi_n(z)}{I_{1,\lambda}(\phi_n)(z)}=\frac{(I_{1,\lambda}(\phi_n))'(z)}{(n+1)I_{1,\lambda}(\phi_n)(z)}=\int\frac{d\nu[I_{1,\lambda}(\phi_n)](w)}{z-w},
$$
and by \eqref{WeakLimit}, $\dsty \lim_{n\to\infty}\frac{\phi_n(z)}{I_{1,\lambda}(\phi_n)(z)}=\int\frac{d\mu_\arc(w)}{z-w}=\psi(z),\quad z\in \Omega.
$ Therefore, the statement (i) of Theorem \ref{Th-IterAsymCompara} follows immediately from the hypothesis \eqref{limitOutArc2} on $(\phi_n)$.
\hfill $\square$

\subsection{Proof of Theorem \ref{Th-IterAsymCompara}(ii): $\lambda \not\in \arc$}

We chose $\lambda_0\in \arc$ and $\lambda\in\Omega$. From \eqref{IntPol}, $I_{m,\lambda}(\phi_n)$   can be written alternatively    as
\begin{equation*}\label{IterativeTaylor}
 I_{m,\lambda}(\phi_n)(z)=I_{m,\lambda_0}(\phi_n)(z)-P_{m-1}(z),
\end{equation*}
where $P_{m-1}(z)$ is the $m-1$-th Taylor polynomial of  $I_{m,\lambda_0}(\phi_n)$ in powers of $(z-\lambda)$. From \eqref{IntPol2},  we have that ${(I_{m,\lambda_0}(\phi_n))^{(k)}(\lambda)}{k!}= \binom{n+m}{ k}  \, I_{m-k,\lambda_0}(\phi_n)(\lambda)$, $0\le k\le m-1$. Therefore,
\begin{equation*}
   P_{m-1}(z)    =\sum_{k=0}^{m-1} \binom{n+m}{k}  \, I_{m-k,\lambda_0}(\phi_n)(\lambda)\, (z-\lambda)^{j}. \label{Taylor_Poly}
\end{equation*}
The proof is completed using (i) of Theorem \ref{Th-IterAsymCompara} for $I_{m-k,\lambda_0}(\phi_n)$.
\hfill$\square$

\begin{figure}[th]\centering
  \includegraphics[width=\textwidth]{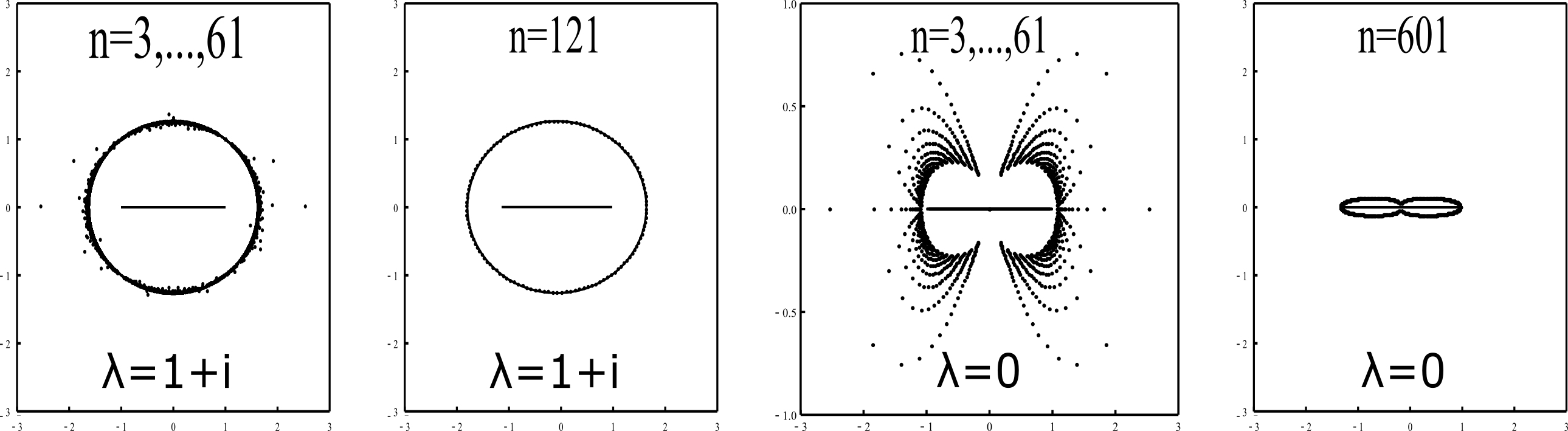}
  \caption{Here we illustrate Theorem \ref{Th-IterAsymCompara} for ultraspherical polynomials on $[-1,1]$: $I_6(P_n^{(3/2)})$.}
  \label{UltrasphLambdaNoArc}
\end{figure}

\begin{figure}[th]\centering
  \includegraphics[width=\textwidth]{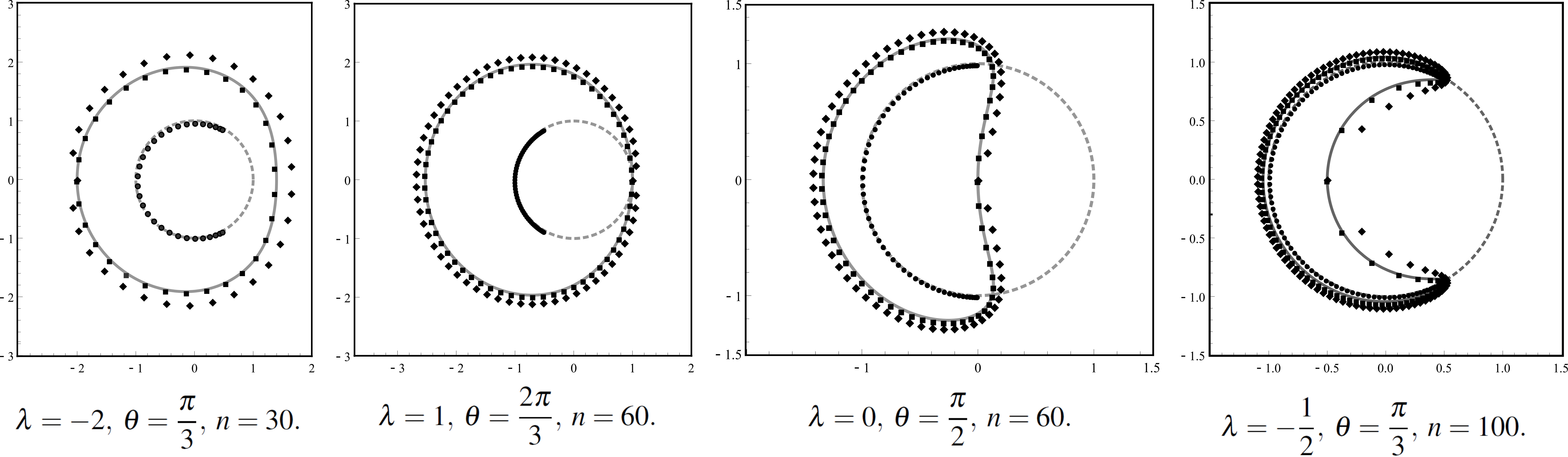}
  \caption{We illustrate Theorem \ref{Th-IterAsymCompara} for orthogonal polynomials on the arc $\{e^{it}:\theta\le t\le 2\pi-\theta\}$ of the unit circle. In each case we take the Chebyshev measure on the arc.  The zeros of $\phi_n$ are marked by bullets, the zeros of $I_{1,\lambda}(\phi_n)$ by square, and the zeros of $I_{2,\lambda}(\phi_n)$ by diamonds. In each case, the dashes circle is $\partial\DD_1$ and the gray path is  the  curve  $\partial\Lambda_\lambda$ for the corresponding values of $\lambda$.}
  \label{bananas}
\end{figure}
By way of Theorem \ref{Th-IterAsymCompara}'s illustration, Figures \ref{UltrasphLambdaNoArc} and \ref{bananas} include several  examples of the curve  $\partial\Lambda_\lambda$, the zeros of the $\phi_n$, and the zeros of their iterated integrals for different values of $\lambda$, $m$.

\section*{Acknowledgements}
The research of H. Pijeira was   partially supported by  Ministry of Science, Innovation and Universities of Spain, under grant  PGC2018-096504-B-C33. The research of M. Bello  was   partially supported by Ministry of Economy and Competitiveness of Spain, under grant MTM2014-54043-P.

\end{document}